\documentclass[12pt]{article}

\usepackage{amsmath}
\usepackage{amsthm}
\usepackage{amsfonts}
\usepackage{amssymb}
\usepackage{mathrsfs}
\usepackage{graphicx}
\usepackage[usenames]{color}
\DeclareMathAlphabet{\mathpzc}{OT1}{pzc}{m}{it}

\parskip \smallskipamount

\newtheorem{lemma}{Lemma}[section]

\newtheorem{proposition}{Proposition}[section]
\newtheorem{definition}{Definition}[section]
\theoremstyle{remark}

\newtheorem{remark}{Remark}[section]

\def\R{\mathbb{R}}

\def\N{\mathbb{N}}

\def\P{\mathbf{P}}
\def\E{\mathbf{E}}

\def\LL{\mathcal{L}}

\def\SS{\mathcal{S}}
\def\GG{\mathcal{G}}
\def\FF{\mathcal{F}}
\def\HH{\mathcal{H}}
\def\DD{\mathcal{D}}
\def\RR{\mathcal{R}}

\def\(({\left(}
\def\)){\right)}

\def\forall{\textrm{ for all }}
\def\for{\textrm{ for }}
\renewcommand{\phi}{\varphi}
\renewcommand{\epsilon}{\varepsilon}

\newcommand{\1}{{\text{\Large $\mathfrak 1$}}}

\newcommand{\beqr}{\begin{eqnarray*}}
\newcommand{\eeqr}{\end{eqnarray*}}
\newcommand{\beq}{\begin{equation}}
\newcommand{\eeq}{\end{equation}}
\newcommand{\bal}{\begin{align*}}
\newcommand{\enal}{\end{align*}}
\renewcommand{\limsup}{\varlimsup}
\renewcommand{\liminf}{\varliminf}
\newcommand{\oo}[1]{\overline{#1}}
\newcommand{\defn}{:=}

\title{On Sums of Conditionally Independent Subexponential Random Variables}
\author{Serguei Foss$^{1}$ and Andrew Richards$^{1}$}
\date{}

\begin{document}

\maketitle
\stepcounter{footnote}\footnotetext{School of Mathematics and Computer Sciences and the Maxwell Institute for Mathematical Sciences, 
  Heriot-Watt University, Edinburgh, EH14 4AS, Scotland, UK.  E-mail: S.Foss@hw.ac.uk and awr2@hw.ac.uk}
  
\begin{quotation}\small

The asymptotic tail behaviour of sums of independent subexponential random variables is well understood, one of the main characteristics being \textit{the principle of the single big jump}.  We study the case of dependent subexponential random variables, for both deterministic and random sums, using a fresh approach, by considering conditional independence structures on the random variables. We seek sufficient conditions for the results of the theory with independent random variables still to hold.  For a subexponential distribution, we introduce the concept of a boundary class of functions, which we hope will be a useful tool in studying many aspects of subexponential random variables.  The examples we give in the paper demonstrate a variety of effects owing to the dependence, and are also interesting in their own right.

\end{quotation}


\section{Introduction}

Finding the asymptotic tail behaviour of sums of heavy-tailed random variables is an important problem in finance, insurance and many other disciplines.  The case when the random variables are independent and subexponentially distributed has been extensively studied and is well-understood.  The key idea is that such a sum will exceed a high threshold because of a single, very large jump; following other authors we shall refer to this as \textit{the principle of the single big jump}.  However, for many practical purposes the independence assumption is too restrictive.  In recent years, many authors have developed results in this area (see, for example, [1,2,4,6,9-13] and references therein).  Denuit, Genest and Marceau \cite{Den} constructed bounds for these sums, but did not consider asymptotics.  Goovaerts, Kaas, Tang and Vernic \cite{Goov} considered the situation of dependent random variables with regularly varying tails; there have also been results on negative dependence for various classes of subexponential distributions (see, for example \cite{Tang1}) and for dependence structures that are \textquoteleft not too positive\textquoteright (see \cite{KoT}).

Once we drop the requirement of independence, two questions naturally  arise.  First, \textit{what kind of behaviours can occur as the dependence between the random variables strengthens?}  And secondly, \textit{how far beyond the independent case does the principle of the single big jump still hold?}  These questions are of real interest, both from theoretical and practical viewpoints.

 Albrecher, Asmussen and Kortschak \cite{AAK} consider the first question for the sum of two dependent random variables.  Their approach, as for many authors, is to study the possible effects of the dependence by considering the copula structure.  They demonstrate that many possible behaviours naturally occur, and that, in some specific cases the principle of the single big jump is insensitive to the strength of the copula structure.  Other papers that concentrate on the copula structure include \cite{ALW, KA}.  Mitra and Resnick \cite{MR} investigate random variables belonging to the maximum domain of attraction of the Gumbel distribution and which are asymptotically independent.  The results we present contain overlap with all these approaches, but we neither impose a particular dependence structure, nor a particular distribution for the random variables, beyond the necessary constraint that at least one be subexponential.

We wish to consider the second question, and to establish conditions on the strength of the dependence which will preserve the results of the theory established for independent random variables; in particular, the principle of the single big jump.  This principle is well known.  However, we would like to examine it again from a probabilistic point of view by considering the sum of two identically distributed subexponential random variables $X_1,X_2$.
\begin{align}
& \P(X_1+X_2>x)  = \P(X_1 \vee X_2 > x) + \P(X_1 \vee X_2 \leq x, X_1+X_2>x)\nonumber\\
& = \P(X_1>x)+\P(X_2>x) - \P(X_1 \wedge X_2>x) + \P(X_1 \vee X_2 \leq x, X_1+X_2>x)\nonumber\\
& \defn \P(X_1>x)+\P(X_2>x) - P_2(x) +P_1(x),
\end{align}
where $X_1 \vee X_2 = \max(X_1,X_2)$ and $X_1 \wedge X_2 = \min(X_1,X_2)$.
If $P_1(x)$ is negligible compared to $\P(X_1>x)$, which in the independent case follows from the definition of subexponentiality, we shall say that we have the \textit{principle of the big jump}.  If in addition $P_2(x)$ is negligible compared to $\P(X_1>x)$, as again is straightforward in the independent case, then we shall say that we have the \textit{principle of the} \textbf{single} \textit{big jump}.  If the dependence is very strong, for instance if $X_1 = X_2$ a.s. (almost surely), then clearly the principle of the single big jump fails.  We shall see in Example 3 a more interesting example where the principle of the big jump holds, but not the principle of the single big jump; but nonetheless a high level is exceeded because of a single big jump a positive fraction of the time. 

We consider sums of random variables that are conditionally independent on some sigma algebra. This is a fresh approach to studying the effect of dependence on subexponential sums and allows a great deal of generality (in particular, we need neither specify a particular subclass of subexponential distribution for which our results hold, nor assume the summands are identically distributed, nor specify any particular copula structure).  We believe this is a fruitful line of enquiry, both practically and theoretically, as the range of examples we give illustrates.

Clearly, any sequence of random variables can be considered to be conditionally independent by choosing an appropriate sigma algebra on which to condition.  This is an obvious observation, and in itself not really helpful.  However, there are practical situations where a conditional independence structure arises naturally from  the problem.  As an example, consider  a sequence of identical random variables $X_1, X_2, \ldots,X_n$, each with distribution function $F_{\beta}$ depending on some parameter $\beta$ that is itself drawn from a different distribution.  The $X_i$ are independent once $\beta$ is known: this is a typically Bayesian situation.  It is natural to view the $X_i$ as conditionally independent on the sigma algebra generated by $\beta$.  We suppose the $X_i$ to have subexponential (unconditional) distribution $F$ and ask under what conditions the distribution of the sum follows the principle of the single big jump.  

In addition to the assumption of conditional independence, we assume that the distributions of our random variables are asymptotically equivalent to multiples of a given reference subexponential distribution (or more generally we can assume weak equivalence to the reference distribution).  This allows us to consider non-identical random variables.  As an example, developed fully in Example 4, which follows ideas in \cite{LGH}, we consider the problem of calculating the \textit{discounted loss reserve}; this can also be viewed as finding the value of a perpetuity.  Let the i.i.d. sequence $X_1,X_2,\ldots,X_n$ denote the net losses in successive years, and the i.i.d sequence $V_1,\ldots,V_n$ denote the corresponding present value discounting factors, where the two sequences are mutually independent.  Then $Y_i = X_i \Pi_{j=1}^i V_j$ represents the present value of the net loss in year $i$, and $S_n = \Sigma_{i=1}^n Y_i$ is the discounted loss reserve.  Conditional on $\sigma(V_1,\ldots,V_n)$ the random variables $Y_i$ are independent.  Let the (unconditional) distribution function of $Y_i$ be $F_i$.  We suppose there is a reference subexponential distribution $F$ and finite constants $c_1, \ldots, c_n$, not all zero, such that, for all $i=1,2,\ldots,n$,
\[\lim_{x \to \infty} \frac{\oo{F_i}(x)}{\oo{F}(x)}=c_i.\]
We seek conditions on the dependence which will ensure that the principle of the single big jump holds for the discounted loss reserve.

More generally, we want to consider both deterministic sums and randomly stopped sums, where the stopping time $\tau$ is independent of the $X_i$ and has light-tailed distribution.

In \cite{FKZ} Foss, Konstantopoulos and Zachary studied time modulated random walks with heavy-tailed increments.  In their proofs of two key theorems (Theorems 2.2 and 3.2) they used a coupling argument involving the sum of two conditionally independent random variables which entailed proving a lemma (Lemma A.2) which considered a particular case of conditional independence.  The investigation in the present paper considers this problem in much greater generality, whilst retaining the flavour of the simple situation in \cite{FKZ}.

The statements of the propositions that we prove in Section 2 therefore hold no surprises, and indeed, once the conditions under which these propositions hold were determined, the proofs followed relatively straightforwardly with no need for complicated machinery.  The interest and effort was in the formulation of the conditions in the first place, which constituted the major intellectual work in this paper, and in finding means by which these conditions could be efficiently checked.

A very useful tool in the study of subexponential distributions is the class of functions which reflect the fact that every subexponential distribution is long-tailed.  For a given subexponential distribution $F$ this is the class of functions $\mathpzc{h}$, such that whenever $h \in \mathpzc{h}$, $h$ is monotonically tending to infinity, and
\beq\label{long-tail}\lim_{x \to \infty} \frac {\oo{F}(x-h(x))}{\oo{F}(x)} = 1.\eeq
A quantity that occurs often in the study of such distributions is
\beq\label{SE1}\int_{h(x)}^{x-h(x)} \oo{F}(x-y)F(dy),\eeq
for any $h$ satisfying \eqref{long-tail}.  The importance of this stems from the fact that this quantity is negligible compared to $\oo{F}(x)$ as $x \to \infty$.  Intuitively this means that, when considering the probability that the sum of two i.i.d. subexponential random variables exceeds some high level $x$, the probability that both of them are of 'intermediate' size is negligible compared to the probability that exactly one of them exceeds $x$.  Because of the form of the integral, it is convenient only to consider those functions $h$ which satisfy $h(x)<x/2$.

It is clear that if $h_1(x)$ belongs to the class $\mathpzc{h}$ then any $h_2(x) \leq h_1(x)$ also satisfies \eqref{long-tail}.  For many long-tailed functions a \textit{boundary class} of functions, $\HH$, exists such that the statement $h \in \mathpzc{h}$  is equivalent to $h(x) = o(H(x))$ for every $H \in \HH$.  This boundary class is particularly useful in dealing with expressions such as that in \eqref{SE1}.  In these types of expression we need to find a suitable function $h$, but the class of functions satisfying the long-tail property \eqref{long-tail} is very rich, and finding an appropriate function can be difficult.  We note that if we have found an appropriate $h_1$ satisfying \eqref{SE1}, then any other $h_2$ in the same class  such that $h_2(x) > h_1(x)$ will also satisfy \eqref{SE1}; we call this \textit{increasing function behaviour}.  We show that in cases where the boundary class exists, for any property that exhibits this increasing function behaviour, the property is satisfied for some function $h \in \mathpzc{h}$ if and only if the property is satisfied for all functions in the boundary class.  Further, since all functions in the boundary class are weakly equivalent (see Section 3 for precise definitions and statements) it suffices to verify such a property only for multiples of a single function.  We hope that this technical tool will be of use to other researchers.

We give a wide range of examples of collections of random variables, some satisfying the principle of the single big jump, some not, and we suggest that these examples are of independent interest in and of themselves.

The paper is structured as follows.  In Section 2 we formulate our assumptions, then state and prove our main results for conditionally independent non-negative random variables satisfying the principle of the single big jump, leaving the more general case of real-valued random variables to Section 5.  In Section 3 we introduce the concept of the Boundary Class for long-tailed distributions, and give some typical examples.  In Section 4 we give examples of conditionally independent subexponential random variables, some of which satisfy the principle of the single big jump, and one of which does not.  In Section 5 we extend our investigation to any real-valued subexponential random variables.  This involves imposing an extra condition.  We give an example that shows that this condition is non-empty and necessary.  Finally, in Section 6 we collect together the different notation we have used, and also give definitions of the standard classes of distributions (heavy-tailed, long-tailed, subexponential, regularly varying, and so on) that we use in this paper.

\section{Main Definitions, Results and Proofs}
  A distribution function $F$ supported on the positive half-line is subexponential if and only if 
\[\oo{F^{*2}}(x) \defn \int_{0}^{x} \oo{F}(x-y)F(dy) + \oo{F}(x) \sim 2\oo{F}(x).\]
It is known (see, for example, \cite{FZ}), and may be easily checked, that a distribution supported on the positive half-line is subexponential if and only if the following two conditions are met:
\begin{enumerate}
\item  $F$ is long-tailed.  That is, there exists a non decreasing function $h(x)$, tending to infinity, such that \eqref{long-tail} holds.\footnote{We observe that, given a random variable $X$ with subexponential distribution $F$, the function $h$ measures how light, compared to $X$, the distribution of a random variable $Y$ must be so that the tail distribution of the sum $X+Y$ is insensitive to the addition of $Y$, not assumed to be independent of $X$.  In particular, if $\P(Y>h(x)) =o(\oo{F}(x))$, then $\P(X+Y>x) \sim \P(X>x)$ regardless of how strong any dependence between $X$ and $Y$ is.  In the case of regular variation, this comment is originally due to C. Kl\"{u}ppelberg \cite{CK}.}  (Examples include: for $F$ regularly varying (see Section 6.2 for definition), then we can choose $h(x)=x^{\delta}$, where $0<\delta<1$; for $F$ Weibull, with parameter $0<\beta<1$, we can choose $h(x)=x^{\delta}$, where $0<\delta<1- \beta$.)
\item For any $h(x)<x/2$ tending monotonically to infinity,
\begin{equation}\label{intermed}
\int_{h(x)}^{x-h(x)}\oo{F}(x-y)F(dy) = o(\oo{F}(x)).
\end{equation}
\end{enumerate}

We work in a probability space $(\Omega,\FF,\P)$.
Let $X_i$, $i=1,2,\ldots$, be non-negative random variables with distribution function (d.f.) $F_i$.
Let $F$ be a subexponential reference distribution concentrated on the positive half-line and $h$ be a function satisfying the long-tailed condition \eqref{long-tail}. Let $\GG$ be a $\sigma$-algebra, $\GG \subset \FF$.  We make the following assumptions about the dependence structure of the $X_i$'s:

\begin{enumerate}
	\item[(D1)]
	$X_1, X_2, \ldots$ are conditionally independent on $\GG$. That is, for any collection of indices $\{ i_1, \ldots , i_r \}$, and any collection of sets $\{B_{i_1}, \ldots , B_{i_r} \}$, all belonging to $\FF$, then $\P(X_{i_1} \in B_{i_1}, \ldots , X_{i_r} \in B_{i_r} | \GG)) = \P(X_{i_1} \in B_{i_1}|\GG )\P(X_{i_2} \in B_{i_2}| \GG )\ldots \P(X_{i_r} \in B_{i_r}| \GG).$
	\item[(D2)]
	For each $i \geq 1$,	$ \oo{F}_i(x) \sim c_i\oo{F}(x)$, with  at least one $c_i \neq 0$, and for all $i \geq 1$ there exists $c>0$ and $x_0>0$ such that $\oo{F}_i(x) \leq c\oo{F}(x)$ for all $x>x_0$.
	\item[(D3)]
	For each $i \geq 1$ there exists a non-decreasing functions $r(x)$ and an increasing collection of sets $B_i(x) \in \GG$, with $B_i(x) \to \Omega$ as $x \to \infty$, such that
\begin{equation}\label{Conditional}	
\P(X_i>x|\GG) \1(B_i(x)) \leq r(x) \oo{F}(x) \1(B_i(x)) \quad \textrm{almost surely}.
\end{equation}	
and, as $x \to \infty$, uniformly in $i$,	
	\begin{enumerate}
	\item[(i)] \begin{equation}\P(\oo{B}_i(h(x)))=o(\oo{F}(x))\end{equation}
	\item[(ii)] \begin{equation}r(x) \oo{F}(h(x)) = o(1),\end{equation}
	\item[(iii)] \begin{equation}r(x) \int_{h(x)}^{x-h(x)} \oo{F}(x-y)F(dy) = o(\oo{F}(x)).\end{equation}
	\end {enumerate}
\end{enumerate}

\begin{remark}
In many cases the the dependence between the $\{X_i\}$ enables us to choose a common $B(x) = B_i(x), \forall i$. However, we allow for situations where this is not the case.  There is no need for a similiar generality in choice of the function $r(x)$ because of the uniformity in $i$.  The function $r(x)$ can be chosen so that it is only eventually monotone increasing, and in the case where we are only considering a finite collection of random variables $\{X_i\}$ it is sufficient to show that the chosen function is asymptotically equivalent to a monotone increasing function.
\end{remark}
\begin{remark}
If the collection of r.v.s of interest is finite, then clearly the uniformity in $i$ needed in conditions (D2) and (D3) is guaranteed.
\end{remark}
\begin{remark}
If the reference distribution $F$ has tail which is intermediately regularly varying (see Section 6.2 for definitions) then it will be shown later that we can check that the conditions (D3) hold for some $h(x)$ satisfying \eqref{long-tail} by verifying that the conditions hold when $h(x)$ is replaced all the functions $H(x)=cx$ where $0<c<1/2$.
\end{remark}
\begin{remark}
It will sometimes be the case that the random variables $X_1,X_2,\ldots$ are not identically distributed, and are not all asymptotically equivalent to the reference distribution $F$.  In these cases it is sufficient to require that they are weakly equivalent to $F$, and that they are subexponentially distributed.  The uniformity condition will still be required.
\end{remark}

\begin{remark}
The need for and the meaning of the bounding functions $r(x)$ and the bounding sets $B_i(x)$ will become apparent when we give some examples.  However, some preliminary comments may assist at this stage.
\begin{itemize}
\item In order to preserve the desired properties from the independent scheme, we need to ensure that the influence of the $\sigma$-algebra $\GG$ that controls the dependence is not too strong.  This we have done by introducing the bounding function $r(x)$ for the $i^{\textrm{th}}$ random variable, which ensures that there are not events in $\GG$ which totally predominate if a high level is exceeded.  Although $r(x)$ may tend to infinity, it must not do so too quickly.
\item Depending on the nature of the interaction of $\GG$ with the random variables, there may be events in $\GG$ that do overwhelmingly predominate when exceeding a high level; this is not a problem as long as these events are unlikely enough and their probability tends to zero as the level tends to infinity.  Within the bounding sets $B_i(x)$ no events in $\GG$ predominate, and we then require that the compliments $\oo{B}_i(x)$ decay quickly enough.
\end{itemize}
\end{remark}

We have the following results.

\begin{proposition}\label{nSum}
Let $X_i$, $i=1,2, \ldots$ satisfy conditions (D1), (D2) and (D3) for some subexponential $F$ concentrated on the positive half-line and for some $h(x)$ satisfying \eqref{long-tail}.  Then
\[\P(X_1+ \cdots + X_n>x) \sim \sum_{i=1}^n \P(X_i>x) \sim \((\sum_{i=1}^n c_i\))\oo{F}(x).\] 
\end{proposition}
\begin{remark}
Lemma A2 in \cite{FKZ} follows directly from this proposition.
\end{remark}

In order to use dominated convergence to generalize Proposition \ref{nSum} to random sums, we need the following extension of Kesten's Lemma.

\begin{lemma}\label{Kesten}
With the conditions of (D1), (D2) and (D3), for any $\epsilon > 0$ there exist $V(\epsilon) > 0$ and $x_0=x_0(\epsilon)$ such that, for any $x > x_0$ and $n \geq 1$,
\[\P(S_n > x) \leq V(\epsilon)(1+\epsilon)^n \oo{F}(x).\]
\end{lemma}

\begin{proposition}\label{random sum}
If, in addition to the conditions of (D1), (D2) and (D3), $\tau$ is an independent counting random variable such that $\E(e^{\gamma \tau})< \infty$ for some $\gamma > 0$, then 
\begin{align*}
\P(X_1+ \cdots + X_{\tau} > x) & \sim \E\((\sum_{i=1}^{\tau} \P(X_i>x)\))\\
& \sim \E\((\sum_{i=1}^{\tau} c_i \right)\oo{F}(x).
\end{align*}
\end{proposition}

Clearly, checking that (D3ii) and (D3iii) hold is the most laborious part of guaranteeing the conditions for these propositions.  Hence we propose a sufficient condition, analogous to the well-known condition for subexponentiality.

\begin{proposition}\label{haz}
Let $F$ be a subexponential distribution concentrated on the positive half-line, $h(x)$ be a function satisfying \eqref{long-tail}, and $r(x)$ a non-decreasing function.  Let $Q(x) \defn -\log(\oo{F}(x))$, the hazard function for $F$, be concave for $x \geq x_0$, for some $x_0<\infty$.  Let
\begin{equation}\label{hazard} x r(x) \oo{F}(h(x)) \to 0 \quad \textrm{as } x \to \infty.\end{equation}
Then conditions (D3ii) and (D3iii) are satisfied.
\end{proposition}

Before we can give examples of our conditions in practice, we need to address a specific issue.  The conditions depend on being able to choose bounding functions $r(x)$ and bounding sets $B_i(x)$, which themselves depend on our choice of the little-$h$ function satisfying \eqref{long-tail}.  The choice of $h(x)$ is not unique, so the fact that one is unable to find appropriate bounding functions and sets for a particular little-$h$ function does not imply that one cannot find them for some other choice of $h$.  This is the problem we address in the next section.

Now we proceed with the proofs of our results.
\def\proofname{Proof of Proposition \ref{nSum}}
\begin{proof}
First consider $X_1+X_2$.  Assume, without loss of generality, that $c_1>0$.  Let $Y$ be a random variable, independent of $X_1$ and $X_2$ with distribution function $F$.  We have the inequalities
\begin{align*}
\P(X_1+X_2>x)  \leq & \P(X_1>x-h(x))+\P(X_2>x-h(x))\\
	& +\P(h(x)<X_1\leq x-h(x),X_2>x-X_1),
	\end{align*}
and
\[\P(X_1+X_2>x)  \geq \P(X_1>x) + \P(X_2>x) - \P(X_1 >x, X_2 >x).\]
Now,
\begin{align*}
& \P(h(x)<X_1\leq x-h(x),X_2>x-X_1)\\
= & \E(\P(h(x)<X_1\leq x-h(x),X_2>x-X_1|\GG))\\
= & \E\(( \int_{h(x)}^{x-h(x)}\P(X_1 \in dy|\GG)\P(X_2>x-y|\GG)(\1(B_2(x-y)) + \1(\oo{B}_2(x-y)))\))\\
\leq & r(x) \E\left( \int_{h(x)}^{x-h(x)}\P(X_1 \in dy|\GG)\P(Y>x-y)\right) + \E(\1(\oo{B}_2(h(x))))\\
= & r(x) \int_{h(x)}^{x-h(x)}\P(X_1 \in dy) \oo{F}(x-y) + o(\oo{F}(x))\\
= & o(\oo{F}(x)).
\end{align*}
Also,
\begin{align*}
& \P(X_1>x,X_2>x)\\
= & \E(\P(X_1>x,X_2>x|\GG)(\1(B_2(x))+\1(\oo{B}_2(x))))\\
\leq & \E(\P(X_1>x|\GG)\P(X_2>x|\GG)\1(B_2(x))) + \E(\1(\oo{B}_2(x)))\\
\leq & r(x)\oo{F}(x)\P(X_1>x) + o(\oo{F}(x))\\
= & o(\oo{F}(x)).
\end{align*}

Hence, $\P(X_1+X_2>x) \sim \P(X_1>x)+\P(X_2>x)$.
Since $c_1>0$, then $\P(X_1>x) + \P(X_2>x) \sim (c_1+c_2) \oo{F}(x)$.

Then, by induction, we have the desired result.
\end{proof}

\def\proofname{Proof of Lemma \ref{Kesten}}
\begin{proof}
The proof follows the lines of the original proof by Kesten.  Again, let $Y$ be a random variable, independent of $X_1$ and $X_2$, and with distribution function $F$.
For $x_0 \geq 0$, which will be chosen later, and $k \geq 1$ put
\[\alpha_k = \alpha_k(x_0) \defn \sup_{x> x_0} \frac{\P(S_k>x)}{\oo{F}(x)}.\]
Also observe that 
\[\sup_{0<x\leq x_0} \frac{\P(S_k>x)}{\oo{F}(x)} \leq \frac{1}{\oo{F}(x_0)} \defn \alpha.\]
Take any $\epsilon > 0$.  Recall that for all $i>0$, $\oo{F}_i(x) \leq c \oo{F}(x)$, for some $c>0$ and for all $x>0$.  Then for any $n>1$
\bal
\P(S_n > x) = & \P(S_{n-1} \leq h(x), X_n > x - S_{n-1})\\
&+ \P(h(x) < S_{n-1} \leq x-h(x), X_n > x-S_{n-1})\\
&+ \P(S_{n-1} > x-h(x), X_n > x-S_{n-1})\\
\equiv & P_1(x) +P_2(x) +P_3(x).
\end{align*}
We bound
\[P_1(x)  \leq \P(X_n > x-h(x))  \leq c L(x_0) \oo{F}(x) \]
and
\[P_3(x)  \leq \P(S_{n-1} > x-h(x))  \leq \alpha_{n-1}L(x_0) \oo{F}(x)\]
for $x \geq x_0$, where $L(x) =  \sup_{y \geq x}\frac{\oo{F}( y-h(y))}{\oo{F}(y)}$.\\
For $P_2(x)$,
\bal
& P_2(x) =  \P(h(x) < S_{n-1} \leq x-h(x), X_n > x-S_{n-1})\\
& =  \E\((\int_{h(x)}^{x-h(x)} \P(S_{n-1} \in dy| \GG)\P(X_n>x-y|\GG)(\1(B_n(x-y))+\1(\oo{B}_n(x-y)))\))\\
& \leq \E\((r_n(x)\int_{h(x)}^{x-h(x)} \P(S_{n-1} \in dy| \GG)\P(Y>x-y)\)) + \P(\oo{B}_n(h(x)))\\
& =  r(x)\int_{h(x)}^{x-h(x)} \P(S_{n-1} \in dy)\P(Y>x-y) + \P(\oo{B}_n(h(x)))\\
& \leq  r(x) \(( \int_{h(x)}^{x-h(x)} \P(Y \in dy)\P(S_{n-1}>x-y) + \P(S_{n-1}>h(x))\P(Y>x-h(x))\))\\
&  \quad \quad + \P(\oo{B}_n(h(x)))\\
& \leq (\alpha_{n-1}+\alpha)r(x) \(( \int_{h(x)}^{x-h(x)} \P(Y \in dy)\P(Y >x-y) + \P(Y>h(x))\P(Y>x-h(x))\))\\
& \quad \quad  + \P(\oo{B}_n(h(x)))\\
& = (\alpha_{n-1}+\alpha)\(( r(x)\int_{h(x)}^{x-h(x)} \frac{\oo{F}(x-y)}{\oo{F}(x)}F(dy) + 
		r(x)\oo{F}(h(x))\oo{F}(x-h(x))\)) + \P(\oo{B}_n(h(x))) .
\end{align*}
We now choose $x_0$ such that, for all $x \geq x_0$,
\bal
\frac{\oo{F}(x-h(x))}{\oo{F}(x)} \leq L(x_0) \leq & 1 + \frac{\epsilon}{4};\\
r(x) \int_{h(x)}^{x-h(x)} \frac{\oo{F}(x-y)}{\oo{F}(x)}F(dy) \leq & \frac{\epsilon}{4};\\
 \oo{F}(h(x)) L(x_0) \leq & \frac{\epsilon}{4};\\
\frac{\P(\oo{B}_n(h(x)))}{\oo{F}(x)} \leq & 1
\end{align*}
which can be done by virtue of the long-tailedness of $F$ and conditions (D3).  We then have that
\[P_2(x) \leq \frac{\epsilon}{2} (\alpha_{n-1}+\alpha) \oo{F}(x) + \oo{F}(x)\]
We therefore have
\begin{align*}
\P(S_n > x) & \leq cL(x_0)\oo{F}(x) + \frac{\epsilon}{2}(\alpha_{n-1}+\alpha) \oo{F}(x) + \oo{F}(x) + \alpha_{n-1}L(x_0)\oo{F}(x)\\
& \leq  R\oo{F}(x) + (1 + \frac{3}{4}\epsilon)\alpha_{n-1} \oo{F}(x),
\end{align*}
for some $0<R<\infty$.  Hence
\[\alpha_n \leq R + (1 + \frac{3}{4}\epsilon) \alpha_{n-1}.\]
Then, by induction we have
\begin{align*}
\alpha_n & \leq \alpha_1 (1+\frac{3}{4}\epsilon)^{n-1} + R \sum_{r=0}^{n-2} (1+\frac{3}{4}\epsilon)^r\\
&  \leq Rn(1+\frac{3}{4}\epsilon)^{n-1}\\
& \leq V(\epsilon)(1+\epsilon)^n,
\end{align*}
for some constant $V(\epsilon)$ depending on $\epsilon$.

This completes the proof.
\end{proof}
\def\proofname{Proof of Proposition \ref{random sum}}
\begin{proof}
The proof follows directly from Proposition \ref{nSum}, Lemma \ref{Kesten} and the dominated convergence theorem.
\end{proof}

Before proving Proposition \ref{haz}, we prove the following Lemma, which was originally used without proof in \cite{DFK}.

\begin{lemma}
Let $F$ be long-tailed and concentrated on the positive real line.  Then there exists a constant $C > 0$ such that for any $b>a>0$, 
\[\int_a^b \oo{F}(x-y)F(dy) \leq C \int_a^b \oo{F}(x-y) \oo{F}(y) dy.\]
\end{lemma}

\def\proofname{Proof}

\begin{proof}
Let $y_0 = a$, $s = [b-a]+1$ and $y_{i} = y_{i-1} + (b-a)/s, \quad i = 1,2, \ldots, s$.  Then $y_s=b$.

There exists a constant $C$ such that for any  $ y > 0$,  $\frac{\oo{F}(y)}{\oo{F}(y +1 )} \leq \sqrt{C} < \infty $ since $F$ is long-tailed.

Then
\begin{align*}
\int_a^b \oo{F}(x-y)F(dy) & = \Sigma_{n=0}^{s-1} \int_{y_n}^{y_{n+1}} \oo{F}(x-y)F(dy)\\
	& \leq  \Sigma_{n=0}^{s-1} \int_{y_n}^{y_{n+1}} \oo{F}(x-y_{n+1})(\oo{F}(y_n)-\oo{F}(y_{n+1}))dy\\
		& \leq \Sigma_{n=0}^{s-1} \int_{y_n}^{y_{n+1}} \sqrt{C} \oo{F}(x-y) \oo{F}(y_n) dy\\
	&	\leq \Sigma_{n=0}^{s-1} \int_{y_n}^{y_{n+1}} C\oo{F}(x-y) \oo{F}(y) dy\\
	& = C \int_a^b \oo{F}(x-y) \oo{F}(y) dy.
\end{align*}
\end{proof}

\def\proofname{Proof of Proposition \ref{haz}}
\begin{proof}
Without loss of generality we may assume that $x_0=0$.  Clearly \eqref{hazard} implies that condition (D3ii) holds.  Since $g$ is concave, the minimum of the sum $g(x-y)+g(y)$ on the interval $[h(x),x-h(x)]$ occurs at the endpoints of the interval.  From the lemma, there exists a constant $C>0$ such that
\begin{align*}
\int_{h(x)}^{x-h(x)} \oo{F}(x-y)F(dy) & \leq C  \int_{h(x)}^{x-h(x)} \oo{F}(x-y)\oo{F}(y)dy\\
& = C \int_{h(x)}^{x-h(x)} \exp(-(g(x-y)+g(y)))dy\\
& \leq C x \exp(-(g(h(x))+g(x-h(x))))\\
& = Cx \oo{F}(h(x))\oo{F}(x-h(x)),
\end{align*}
and so
\begin{align*}
r(x) \int_{h(x)}^{x-h(x)} \frac{\oo{F}(x-y)}{\oo{F}(x)}F(dy) & \leq C x r(x) \oo{F}(h(x)) \frac{\oo{F}(x-h(x))}{\oo{F}(x)}\\
&  = o(1).
\end{align*}
Therefore condition (D3iii) also holds.
\end{proof}

\def\proofname{Proof}

\section{The Boundary Class}

The conditions (D3i), (D3ii) and (D3iii) depend on being able to choose a suitable little-$h$ function in $\mathpzc{h}$ .  There is a problem in that being unable to find a suitable function $r(x)$ and bounding sets $B_i(x)$ for a particular little-$h$ function does not mean that these objects cannot be found for some other little-$h$ function.  The class $\mathpzc{h}$ is usually very rich, and so it may be difficult to find a suitable little-$h$ function.  However, if $h_1(x)$ belongs to the appropriate class of functions and satisfies the three conditions in (D3), and $h_2$, belonging to the same class, is such that $h_2(x) > h_1(x)$ for all $x>x_0$, for some $x_0$, then $h_2(x)$ also satisfies the three (D3) conditions.

This property of the (D3) conditions, which we refer to as an increasing function property (defined precisely below), allows us to construct a boundary class of functions.  Functions in the boundary class do not satisfy the little-$h$ condition, but any function that is asymptotically negligible with respect to any function in the boundary class will be in $\mathpzc{h}$.  We show that all functions in the boundary class are weakly equivalent (written $H_1(x) \asymp H_2(x)$, see notations in Section 6.1).  This means that the boundary class can be generated by a single function $H(x)$, and all multiples of $H(x)$.

We will show that if the (D3) conditions are satisfied for all multiples of $H(x)$ then they are satisfied for some $h(x)$ belonging to the little-h class of functions, \textit{without having to find the little-h function}.  A generator for the boundary class is usually easy to find, and almost trivial for absolutely continuous distributions.

First we define precisely what we mean by an increasing (or decreasing) function property.

\begin{definition}
A property depending on a function $h$, where $h$ belongs to some class of functions, is said to be an increasing (decreasing) function property if when the property is satisfied for $h_1$ then it is satisfied by any other function $h_2$ in the class such that $h_2(x) > h_1(x)$, ($h_2(x) < h_1(x)$), for all $x>x_0$, for some $x_0$.
\end{definition}

We observe that the property of long-tailedness \eqref{long-tail} is a decreasing function property.  We want to describe its upper boundary, when it exists.  First, it is convenient to give a name to the space of functions that the little-$h$ function may lie in.  Let $\mathpzc{h}$ be the space of non-decreasing functions $h(x)$ defined on the positive reals such that $0<h(x)<x/2$ and $h(x) \to \infty$ as $x \to \infty$.

\begin{definition}
The \emph{boundary class} (for $F$), $\HH$, consists of all continuous, non-decreasing functions $H(x)$ such that $h(x) \in \mathpzc{h}$ satisfies the long-tail property \eqref{long-tail} (for $F$) if and only if $h(x) = o(H(x))$.
\end{definition}

\begin{remark}  In most cases the boundary class does exist, however we note that slowly varying functions do not possess a boundary class as all functions in $\mathpzc{h}$ can act as little-$h$ functions satisfying \eqref{long-tail}.
\end{remark}

We examine the structure of the boundary class $\HH$, and show that all functions in $\HH$ are weakly tail equivalent.

\begin{proposition}
Let $H_1(x)$ belong to the boundary class $\HH$.  Then $H_2(x) \in \HH$ if and only if $H_2(x) \asymp H_1(x)$.
\end{proposition}
\begin{proof}
Clearly, if $H_2(x) \asymp H_1(x)$ then $H_2(x) \in \HH$.
So, consider a function $H_2(x)$ for which $\liminf \frac{H_2(x)}{H_1(x)} = 0$.  We shall construct a function $h_1(x)$ with the long-tail property \eqref{long-tail} which is not $o(H_2(x))$.  There exists a sequence, tending to infinity, $0=x_0<x_1<\ldots$ with
$ \epsilon_n \defn \frac{H_2(x_n)}{H_1(x_n)}$
such that $\lim_{n \to \infty} \epsilon_n =0$.

By the continuity of $H_1(x)$ and $H_2(x)$ we can find a sequence of points $0<y_1<x_1 <y_2\ldots$ such that $0<\frac{H_2(x)}{H_1(x)} < 2\epsilon_n$ for all $y_n \leq x < x_n$.

For $n \geq 1$, define 
\[ h_1(x) = \left\{ \begin{array}{cc} H_2(x_{n-1}) & \for x \in [x_{n-1},y_n),\\ 
				\frac{x_n-x}{x_n-y_n}H_2(x_{n-1})+\frac{x-y_n}{x_n-y_n}H_2(x_{n}) &  \for x \in [y_n, x_n).
\end{array} \right. \]

Now, $h_1(x)=o(H_1(x)$ by construction, so $h_1(x)$ satisfies the long-tail property, but $\liminf \frac{h_1(x)}{H_2(x)} = 1$, so $H_2(x)\notin \HH$.
We can clearly repeat this argument if $\liminf \frac{H_1(x)}{H_2(x)} = 0$.
Hence, if $H_2(x) \in \HH$ then $H_2(x)\asymp H_1(x)$.
\end{proof}

\begin{remark}
The three conditions in (D3) depend on a little-$h$ which satisfies \eqref{long-tail}.  The next Proposition will show that if the three conditions (D3) are satisfied by \emph{all} functions in the boundary class $\HH$, then they are satisfied by at least one $h$ in the long-tail class.  Since all functions in $\HH$ are weakly-tail equivalent, it will then be sufficient to show that the conditions (D3) hold for all multiples $\{cH(x); c \in \R^+, cH(x) < x/2\}$ of any particular function $H \in \HH$.  We shall then say that $H(x)$ generates the boundary class $\HH$.
\end{remark}

\begin{proposition}
Let $F$ be a distribution function possessing a boundary class $\HH$.  Then there exists some function $h(x)$, which itself satisfies the long-tail property \eqref{long-tail} if and only if, for any $H(x) \in \HH$, the three conditions in (D3) hold for $\{cH(x); c \in \R^+, cH(x) < x/2\}$ .
\end{proposition}
\begin{proof}
Let 
\[\Pi(f(x))=\sup_i \left( \max \left( \frac{\P(\oo{B_i}(f(x)))}{\oo{F}(x)}, r(x) \oo{F}(f(x)), \int_{f(x)}^{x-f(x)} \frac{\oo{F}(x-y)}{\oo{F}(x)} F(dy) \right) \right) .\]
Choose any $H(x) \in \HH$, and let $c_n = 2^{-n}, \quad n \in \N$.

Define an infinite sequence $0=y_1 < x_1 < y_2 < \ldots$ recursively, for $r \in \N$, by
\beqr
& y_1  = 0;\\
& x_r  = \max(y_r+1, \sup_{x>0} \{ x: \Pi(c_rH(x)) > c_r\});\\
& y_{r+1}  = \inf_{x>x_r+1} \{ x: H(x) = 2H(x_r)\}.
\eeqr
By construction this sequence tends to infinity.

For $x \geq 0$, define
\[h(x) = \left\{ \begin{array}{cc} c_rH(x) & \for x \in [y_r,x_r),\\ c_rH(x_r) &  \for x \in [x_r, y_{r+1}).
\end{array} \right. \]
Hence, if conditions (D3) hold for all (sufficiently small) multiples of $H(x)$, then they hold for $h(x)$, which, by construction, is $o(H(x))$.  Conversely, if conditions (D3) hold for some $h(x)$, then they hold for any function $g(x)$ such that $h(x) = o(g(x))$, and hence for all functions in class $\HH$.
\end{proof}

Since the properties in (D3i), (D3ii) and (D3iii) are all increasing function properties, rather than finding a function $h$, it suffices to check them for all multiples of $H$, for any $H \in \HH$, the boundary class.  This is a much easier proposition than identifying a suitable little-$h$ function.

\begin{remark}
Let $F$ be a distribution function such that $\oo{F}(x) =f_1(x)f_2(x)$, where each of $f_1$ and $f_2$ are long-tailed.  Let the boundary class for $f_i$ be $\HH_i$, $i=1,2$ and generated by $H_i(x)$ respectively.  Assume that $H_2(x) = o(H_1(x))$. Then the boundary class for $F$ is $\HH_2$.  If $f_1(x)$ is slowly varying, then the boundary class for $F$ is again $\HH_2$.
\end{remark}

\begin{proposition}
Let $F$ be an absolutely continuous long-tailed distribution function with density $f(x)$ and hazard rate $q(x) = \frac{f(x)}{\oo{F}(x)}$.   Let $H(x) = 1/q(x)$.  Then the boundary class of $F$ is generated by $\{cH(x); c \in \R^+, cH(x) < x/2\}$
\end{proposition}
\begin{proof}
By the long-tailed property \eqref{long-tail} as $x \to \infty$ and for appropriate $h(x)$ we know that $\oo{F}(x-h(x))=\oo{F}(x) + o(\oo{F}(x))$.  Hence, $x-h(x)=\oo{F}^{-1}(\oo{F}(x)(1+o(1)))$, and since  $\oo{F}^{-1}$ has a derivative at all points in its domain, $h(x)=o(\oo{F}(x)(-\oo{F}^{-1})'(\oo{F}(x)))$, where the negative sign has been introduced to make the function inside the little-$o$ positive.  However, $\oo{F}(x)(-\oo{F}^{-1})'(\oo{F}(x)) = 1/q(x)$.

Conversely, if $h(x)=o(1/q(x))$ then it is easy to show that the long-tailed property \eqref{long-tail} holds.
\end{proof}

We give some examples of calculating the Boundary class.

\begin{itemize}
\item[3.1. ] Let $\oo{F}(x) = l(x)x^{-\alpha}$, $x>1$, where $l(x)$ is slowly varying and $\alpha > 0$.  The boundary class for $f_1(x)=l(x)$ is the whole space of functions.  For $f_2(x)=x^{-\alpha}$ we have $q_2(x)=\alpha/x$.  Hence $\HH=\{cx; 0<c<1/2\}$.\\
\item[3.2. ] Let $\oo{F}(x)=\exp(-\gamma x^{\beta})$, $x>0$, where $0<\beta<1$.  Then $q(x)=\gamma \beta x^{-1+\beta}$, and $\HH=\{cx^{1-\beta};c>0\}$.\\
\item[3.3. ] Let $\oo{F}(x)= f_1(x)\exp(-\gamma (\log(x))^{\alpha}) \defn f_1(x)f_2(x)$, $x>1$, where $\alpha \geq 1$ and $f_2(x)=o(f_1(x))$.  We note that this class of functions includes regular variation as in Example 3.1. above, and also log-normal.  We need only consider $f_2(x)$.  We then find that $\HH=\{cx(\log(x))^{1-\alpha};c>0\}$.
\end{itemize}

\section{Examples of Conditionally Independent Subexponential Random Variables}

\subsection{Example 1} Let $\xi_i$, $i=1,2, \ldots, n$, be i.i.d. with common distribution function $F_{\xi} \in R_{-\alpha}$ (for definition of the class of regularly varying distributions of degree $\alpha$ see Section 6.2).  Let $\eta$ be independent of the $\xi_i$ and have distribution function $F_{\eta} \in R_{-\beta}$, where $\alpha \neq \beta$.  Define $X_i=\xi_i+\eta$ for $i=1,2, \ldots, n$, and let the reference distribution be $\oo{F}(x) = x^{-(\alpha \wedge \beta)}$.  Then, from well-known properties of independent subexponential random variables, we have, for $i=1,2, \ldots, n$,
\[\P(X_i>x) \sim \oo{F}(x).\] 

Conditional on the sigma algebra $\GG= \sigma(\eta)$, the $X_i$ are independent.

For our reference distribution, the boundary class is $\HH=\{cx,c>0\}$.
Now, the random variables $\frac{\P(X_i>x|\GG)}{\oo{F}(x)} \leq \frac{1}{\oo{F}(x)}$ are unbounded as $x \to \infty$.  If we try to satisfy the condition \eqref{Conditional} without using the bounded sets $B(x)$, we need to ensure that, for all $x>0$, almost surely,
\[\frac{\P(X_i>x|\GG)}{\oo{F}(x)} \leq r(x).\]
If we take $r(x) = \frac{1}{\oo{F}(x)}$, then condition (D3ii) is not satisfied, since for any $c>0$,
\[r(x)\oo{F}(H(x))= c^{-(\alpha \wedge \beta)}.\]

Hence, we need to use bounded sets.  Let $B(x)=\{\eta \leq x/2\}$.  This satisfies condition (D3i), if and only if $\alpha < \beta$:  for any $c>0$,
\[\frac{\P(\oo{B}(H(x))}{\oo{F}(x)}=\frac{ \P(\eta>cx/2)}{\oo{F}(x)}=(2/c)^\beta x^{(\alpha \wedge \beta)-\beta}.\]
Clearly, in the case $\alpha < \beta$, we may take $r(x)$ as a constant, $r(x)=2^{\beta}$.

The condition $\alpha < \beta$ agrees with arguments on $X_i$ taken from the standard theory of independent subexponential random variables.

\subsection{Example 2} Let $\eta$ be a random variable with uniform distribution in the interval$(1,2)$.  Conditional on $\GG= \sigma(\eta)$ let $X_i$, $i=1,2, \ldots, n$, be i.i.d. with common distribution function 
\[\oo{F}_{\xi|\eta}(x) = (1+x)^{-\eta}, \quad x > 0.\]

Routine calculations show that
\[\P(X_i>x) \sim \frac{1}{x\log(1+x)} \equiv \oo{F}(x),\]
where $F$ is our reference distribution.  The boundary class is again $\HH=\{cx,c>0\}$.

For all $x>1$ we have, almost surely,
\begin{align*}
\frac{\P(X_i>x|\GG)}{\oo{F}(x)} & \leq \frac{\P(X_i>x|\eta=1)}{\oo{F}(x)}\\
& = (1+1/x)\log(1+x)\\
& \leq r(x) \equiv 2 \log(1+x)
\end{align*}
Routine calculations show that, for all $0<c<1/2$ condition (D3iii) is satisfied, and also condition (D3ii).

In this example there has been no need to define bounding sets, or equivalently we can take $B(x) = \Omega$ for all $x>0$.

\subsection{Example 3} In this example we again consider a Bayesian type situation.  Let $X_i$, $i=1,2, \ldots, n$ be identically distributed, conditionally independent on parameter $\beta$, with conditional distribution $F_{\beta}$ given by
\[\oo{F}_{\beta}(x) = \exp(-\gamma x^{\beta}), \quad \gamma >0\]
where $\beta$ is drawn from a uniform distribution on $(a,b)$, $0 \leq a < 1$, $a<b$.
The unconditional distribution of $X_i$, $F_X$, is then
\[\oo{F}_X(x)= \frac{1}{(b-a)\log x} \((E_1(\gamma x^a)-E_1(\gamma x^b)\)),\]
where $E_1(x) = \int_x^{\infty} \frac{e^{-u}}{u}du.$

We now consider separately the two cases (i) $0<a<1$ and (ii) $a=0$.  We start first with case (i).

We find 
\[\oo{F}_X(x) \sim \oo{F}(x) \defn \frac{\exp(- \gamma x^a)}{(b-a)\gamma x^a \log x}.\]
The boundary class $\HH$ is the same as for the Weibull distribution, 
\[\HH=\{cx^{1-a}, \: c>0\}.\]
We shall take $B(x) = \Omega$ for all $x>0$, and
\[r(x)=\gamma (b-a)x^a \log x.\]
Note that $g(x) = -\log\((\oo{F}(x)\)) = \log\((\gamma(b-a)\)) + x^a + a \log x + \log \log x$ is convex, and that
\[x r(x) \oo{F}(H(x)) = \frac{x^{1-a^2} \log(x)\exp\((-\gamma c^a x^{a(1-a)}\))}{ c^a \log(cx)} \to 0\]
as $x \to \infty$ for all $c>0$.
Hence, by Lemma \ref{haz}, the conditions in (D3) are met, and the principle of the single big jump holds.

W now consider case (ii), with $\beta$ distributed uniformly on the interval $(0,b)$.  The reference distribution is now
\[\oo{F}(x) = \frac{E_1(1)}{b \log x} \defn \frac{k}{\log x},\]
where $E_1(1) \approx 0.21938 \ldots$.  

Since $\oo{F}$ is slowly varying there is no Boundary class, but the class of functions satisfying the long-tailed property \eqref{long-tail} is $\{h(x)=O(x)\}$. Therefore, in order to satisfy (D3ii) we need to choose $r(x)$ such that $\lim_{x \to \infty} \frac{r(x)}{\log x} = 0$.
For the bounded sets $B(x)$, the problems clearly occur near $\beta=0$, so we may try sets of the form $B(x)=\{ \beta \in (a(x),b)\}$, with $a(x) \to 0$ as $x \to \infty$. 

To satisfy \eqref{Conditional} we need, for each $x>0$ and for $\eta \in (a(x),b)$,
\[\exp(-\gamma x^{a(x)}) < \exp(-\gamma x^{\beta}) \leq \frac{k r(x)}{\log(x)},\]
which cannot be true if both $a(x) \to 0$ and $\frac{r(x)}{\log x} \to 0$ as $x \to \infty$.

Hence the conditions of (D3) cannot be met.

The question now arises whether, in this case, the principle of the single big jump still holds.  The answer is no.  To see why, we again consider the representation (1), for the sum of two independent identically distributed subexponential random variables $X_1,X_2$.  For simplicity we shall consider the case where $b=\gamma=1$.

Considering the representation in (1), we have
\[P_1(x)= \int_0^1 d \beta \int_0^x \beta u^{\beta-1} \exp(-u^{\beta}) du \int_0^x \beta v^{\beta-1} \exp(-v^{\beta}) dv \1(u+v>x).\]
Making the substitution $u=xy, v=xz$, we have
\begin{align*}
P_1(x)= & \int_0^1 d \beta x^{2 \beta} \int_0^1 \int_0^1 \beta^2 y^{\beta-1} z^{\beta-1}\exp(-x^{\beta}(y^{\beta}+z^{\beta})) \1(y+z>1)dydz\\
\leq & \int_0^1 d \beta x^{2 \beta} \exp(-x^{\beta}) \int_0^1 \int_0^1 \beta^2 y^{\beta-1} z^{\beta-1} \1(y+z>1)dydz\\
= & \int_0^1 d \beta x^{2 \beta} \exp(-x^{\beta}) J(\beta),
\end{align*}
where $J(\beta) = \P(Y_1^{1/\beta} + Y_2^{1/\beta}>1)$ and $Y_1,Y_2 \sim U(0,1)$ are i.i.d..
As $\beta \to 0$, $J(\beta) \to 0$, so for any $\delta>0$ there exists $\epsilon>0$ such that $J(t) \leq \delta$ for all $t \leq \epsilon$.  Hence,
\begin{align*}
P_1(x) \leq & \left( \int_0^{\epsilon} + \int_{\epsilon}^1\right) d \beta x^{2 \beta} \exp(-x^{\beta}) J(\beta)\\
\leq & \int_0^{\epsilon} d \beta x^{2 \beta} \exp(-x^{\beta}) \delta + o(\exp(-x^{\epsilon/2})).
\end{align*}
But
\begin{align*}
 \int_0^{\epsilon} d \beta x^{2 \beta} \exp(-x^{\beta}) \leq & \int_0^1 d \beta x^{2 \beta} \exp(-x^{\beta})\\
= & \frac{1}{\log(x)} \int _1^{\infty} t \exp(-t)dt \\
= & \frac{1}{2\log(x)}.
\end{align*}
Therefore
\[P_1(x) =o(\oo{F}(x)),\]
and the principle of the big jump holds.
However,
\[P_2(x)= \int_0^1 \exp(-2x^{\beta})d \beta = \frac{1}{\log(x)}\int_2^{2x} \frac{e^{-u}}{u}du\]
so that
\[\P(X_1>x,X_2>x) \sim \frac{E_1(1)}{E_1(2)}\P(X_1>x).\]
Hence, the principle of the \textit{single} big jump does not hold.  We note that this rsult is related to Theorem 2.2 in \cite{AAK}.

\subsection{Example 4}

For $i = 1,2,\ldots,n$ let $X_i = \xi_i \eta_1 \eta_2 \ldots \eta_i$, where the $\{\xi_i\}$ are i.i.d, and the $\{\eta_i\}$ are i.i.d. and independent of the $\{\xi_i\}$.  Then conditional on the $\sigma$-algebra generated by $\{\eta_1, \ldots,\eta_n\}$ the $\{X_i\}$ are independent.  Let the $\{\xi_i\}$ have common distribution function $F$ in the intermediately regularly varying class, $F_{\xi} \defn F \in \textrm{IRV}$, (see Section 6 for definition), and let the $\{\eta_i\}$ have common distribution function $F$ that is rapidly varying, $F_{\eta_1} \in \RR_{- \infty}$, (see Section 6 for definition) .  This is related to the example given in \cite{LGH}.  In their example the $\{\xi_i\}$ were chosen to belong to the class $\DD \cap \LL$ (again, see Section 6).  We have chosen the slightly smaller class of intermediate regular variation because:
\begin{enumerate}
\item examples which lie in the $\DD \cap \LL$ class that do not lie in the IRV class are constructed in an artificial manner;\\
\item the IRV class of functions has a common boundary class, and hence is suitable for general treatment under our methodology.
\end{enumerate}
The boundary class for $F$ is $\HH = \{cx, 0<c<1/2\}$.

By Lemma 6.1 the class  $\RR_{- \infty}$ is closed under product convolution, hence for each $i = 1,2,\ldots,n$ we have $X_i$ is of the form $X_i = \xi_i \eta$ where the d.f. of $\eta$, $F_{\eta} \in \RR_{-\infty}$.  Then by Lemma 6.2 each $X_i$ has d.f. $\oo{F}_{X_i}(x) \asymp  \oo{F}(x)$.

As we noted in Remark 2.3, the results of our propositions follow through with the asymptotic condition in (D2) replaced with weak equivalence.

We now proceed to the construction of the bounding sets, $B(x)$.  For \eqref{Conditional} to hold we need to restrict the size of $\eta$.  By Lemma 6.3 we can choose $\epsilon>0$ such that $\oo{F_{\eta}}(x^{1-\epsilon}) = o(\oo{F}(x))$. For such an $\epsilon$  we choose $B(x) = \{\eta \leq x^{1-\epsilon}\}$.  Then for any $H(x)= cx \in \HH$, $0<c<1/2$, condition (D3i) requires
\[ \P(\oo{B}(H(x)) = \oo{F}_{\eta}((cx)^{1-\epsilon}) = o(\oo{F}_{\eta}(x^{1-\epsilon}) = o(\oo{F}(x)),\]
as required.

Now consider \eqref{Conditional}:
\[ \P(X_i>x|\GG) \1 (B(x)) \leq r(x) \oo{F}(x) \1 (B(x)).\]
This implies that the choice for $r(x)$ satisfies
\[ \frac{\P(\xi_i>x/\eta | \eta \leq x^{1-\epsilon})}{\oo{F}(x)} \leq \frac{\P(\xi_i > x^{\epsilon})}{\oo{F}(x)} = \frac{\oo{F}(x^{\epsilon})}{\oo{F}(x)} \leq r(x).\]

Taking $r(x)=\frac{\oo{F}(x^{\epsilon})}{\oo{F}(x)}$, for any $H(x)= cx \in \HH$,
\[r(x)\oo{F}(H(x)) = \frac{\oo{F}(x^{\epsilon})}{\oo{F}(x)}\oo{F}(cx) = o(1),\]
and
\[r(x) \int_{cx}^{(1-c)x} \frac{\oo{F}(x-y)}{\oo{F}(x)} F(dy) \leq \frac{\oo{F}(x^{\epsilon})\oo{F}(cx)}{\oo{F}^2(x)}\oo{F}(cx) = o(1).\]

Hence all the conditions of (D3) are met, and the principle of the single big jump holds; that is:
\[ \P(X_1+ \cdots X_n > x) \sim \sum_{i=1}^n \oo{F}_{X_i}(x).\]

\begin{remark}
If, in addition, $F$ is continuous, then Theorem 3.4 (ii) of \cite{CS} shows that the restriction $F_{\eta} \in \RR_{- \infty}$ can be eased to $\oo{F}_{\eta} =o(\oo{F})$.
\end{remark}

\subsection{Example 5}
Let $\{Z_i\}_{1 \leq i \leq n}$ be independent and normally distributed $Z_i \sim N(\mu_i,\sigma_i^2)$ and let $W$ be independent of $\{Z_i\}$ with standard normal distribution $W \sim N(0,1)$.  Let $Y_i = s_iW+ Z_i$ with $s_i \in \RR, \max_i |s_i | \defn s$.  Let $\eta=e^W$, $\eta_i=e^{s_iW}$, $\xi_i= e^{Z_i}$ and $X_i=e^{Y_i}=\eta_i \xi_i$.  We can think of the $X_i$ as dependent risks, dependent through a response to the market determined by $W$. It is straightforward that $X_i \sim LN(\mu_i, s_i^2+\sigma_i^2)$.  

Let  $\mu = \max_i \mu_i$, $\sigma_Z^2 = \max_i \sigma_i^2$, and $\sigma^2 = s^2 + \sigma_Z^2$.  

Let the reference distribution $F$ have  the $LN(0, \sigma^2)$ distribution, and note that asymptotically $\oo{F}(x) \sim \frac{\sigma}{\sqrt{2\pi}\log(x)} \exp\((- \frac{1}{2 \sigma^2}(\log(x)-\mu)^2\))$ and that the boundary class $\HH$ is generated by $H(x) = x/\log x$.  

Also, let $F_{\xi}$ be the $LN(\mu, \sigma_Z^2)$ distribution. Then if  $X_i$ is such that $\sigma_i = \sigma_Z, \mu_i = \mu$ then, in condition (D2) $c_i = 1$, otherwise $c_i=0$.

The $X_i$ have a lognormal distribution with Gaussian copula, as studied in \cite{ANR}.  From our perspective we choose $\GG$ to be the sigma algebra generated by $\eta$, and dependent on this the $X_i$ are conditionally independent.

Bounding sets will need to be defined since, for instance, if $\eta$ is too large, then any $X_i$ for which $s_i>0$ will also be large.  We choose, for each $i$,
\[B_i(x) = \left\{ \begin{array}{cc} \{\eta \leq x^{1- \delta}\}, & s_i>0,\\
																			\Omega, & s_i = 0,\\
																			\{\eta \geq x^{-1+ \delta}\}, & s_i<0,
\end{array} \right. \]
where $1 > \delta > \max\(( s / \sigma, 1 - \sigma_Z^2 / \sigma^2\))$.  The condition that $\delta >  s / \sigma $ ensures that $\P(\oo{B_i}(H(x)))) = o(\oo{F}(x))$ for all $i$.

We choose the bounding function $r(x) = \frac{\oo{F_{\xi}}(x^{1-\delta})}{\oo{F}(x)}$, and the condition $\delta > 1 - \sigma_Z^2 / \sigma^2$ ensures that $r(x)$ is eventually monotonically increasing.

The conditions for Proposition 2.3 are easily checked, and hence all (D3) conditions are met.  The principle of the single big jump holds and
\[\P(X_1+ \cdots + X_n>x) \sim \((\sum_{i=1}^n c_i\))\oo{F}(x),\] 
where $\sum_{i=1}^n c_i$ counts the number of the $X_i$ that have the heaviest lognormal distribution.

\section{Real-valued Random Variables}

We wish to extend our investigation beyond non-negative random variables to conditionally independent subexponential random variables taking real values.  In order to deal with this situation we need to add another condition to those enumerated in Section 2 at (D1), (D2) and (D3). Again we let $F$ be a reference subexponential distribution, and $h$ be a function satisfying \eqref{long-tail}.
\begin{enumerate}
\item[(D4)] For each $i,j \geq 1$ we have that
\[\P\((X_i>x+h(x), X_j \leq -h(x)\))=o(\oo{F}(x)).\]
\end{enumerate}
We then have the following extension of Proposition 2.1.

\begin{proposition}\label{Genrv}
Let $X_i$, $i=1,2, \ldots$ be real-valued random variables satisfying conditions (D1), (D2), (D3) and (D4) for some subexponential $F$ concentrated on the positive half-line and for some $h(x)$ satisfying \eqref{long-tail}.  Then
\[\P(X_1+ \cdots + X_n>x) \sim \sum_{i=1}^n \P(X_i>x) \sim \((\sum_{i=1}^n c_i\))\oo{F}(x).\] 
\end{proposition}

\begin{proof}
The proof follows the general outline of the proof of Proposition 2.1.
The derivation of an upper bound for $\P(X_1+X_2>x)$ remains as in Proposition 2.1.  For the lower bound we have:
\begin{align*}
\P(X_1+X_2>x) \geq & \P(X_1>x+h(x),X_2>-h(x)) + \P(X_2>x+h(x), X_1>-h(x))\\
& - \P(X_1>x+h(x),X_2>x+h(x))\\
= & \P(X_1>x) + \P(X_2>x) + o(\oo{F}(x)),
\end{align*}
where we have used the long-tailedness of $X_1$ and $X_2$, condition (D4) and $\P(X_1>x+h(x),X_2>x+h(x)) = o(\oo{F}(x))$ for the same reasons as in Proposition 2.1.\\
Hence $\P(X_1+X_2>x) \sim \P(X_1>x)+\P(X_2>x)$, and the rest of the proof follows by induction.
\end{proof}

We can see that the only reason for condition (D4) is to deal with the lower bound, and hence no change is needed to show the two following generalisations of Lemma \ref{Kesten} and Proposition \ref{random sum}.

\begin{lemma}\label{Kesten2}
We let $\{X_i\}$ be as in Proposition \ref{Genrv} and satisfying (D1), (D2), (D3) and (D4).  Then, for any $\epsilon > 0$, there exist $V(\epsilon) > 0$ and $x_0=x_0(\epsilon)$ such that, for any $x > x_0$ and $n \geq 1$,
\[\P(S_n > x) \leq V(\epsilon)(1+\epsilon)^n \oo{F}(x).\]
\end{lemma}

\begin{proposition}\label{random sum2}
If, in addition to the conditions of Lemma \ref{Kesten2}, $\tau$ is an independent counting random variable such that $\E(e^{\gamma \tau})< \infty$ for some $\gamma > 0$, then 
\[\P(X_1+ \cdots + X_{\tau} > x) \sim \E\((\sum_{i=1}^{\tau} \P(X_i>x)\)).\]
\end{proposition}

To show that condition (D4) is both non-empty and necessary we construct an example where it fails to hold and the principle of the single big jump fails.

\newpage

\textbf{Example 6}\\
Consider a collection of non-negative i.i.d random variables $\{Z_i\}_{i \geq 0}$ such that each $Z_i$ has the distribution of a generic independent non-negative random variable $Z$, and $\P(Z>x)=1/x^{\alpha}$ for $x \geq 1$.  Also consider a collection, independent of the $Z_i$, of non-negative i.i.d random variables $\{Y_i\}_{i \geq 0}$ such that each $Y_i$ has the distribution of a generic independent non-negative random variable $Y$, and $\P(Y>x)=1/x^{\beta}$ for $x \geq 1$, where $\alpha>\beta>1$.  For $i \geq 1$ let $X_i=Z_i - Y_iZ_{i-1}.$

First we show that the $X_i$ satisfy conditions (D1), (D2) and (D3).

For any $i \geq 1$ we have that $\P(X_i>x) \sim \oo{F}(x) \defn 1/x^{\alpha}$ for $x \geq 1$, and we recall that the boundary class for $F$ is $\HH=\{cx;0<c<1/2\}$.  We take $\GG= \sigma(\{Y_iZ_{i-1}\}_{i \geq 1})$, and then, conditional on $\GG$, the $X_i$ are independent.  To show that condition (D3) is met we need to consider the random variables $\P(X_i>x|\GG)$, so consider
\[\P(X_i>x|Y_{i+1}Z_i=w) \leq \P(Z_i>x|Y_{i+1}Z_i=w) = \P(Z>x|YZ=w).\]
We calculate that
\begin{align*} \frac{\P(Z>x|YZ=w)}{\oo{F}(x)} & = \left\{	\begin{array}{cc}  
x^{\alpha} \((\frac{x^{-\alpha+\beta}-w^{-\alpha+\beta}}{1-w^{-\alpha+\beta}}\)) & \textrm{for  } 1<x\leq w,\\
	0 &  \textrm{for  } x>c. \end{array} \right.\\
	& \leq x^{\beta} \defn r(x).
	\end{align*}
	Clearly $r(x)\oo{F}(cx)=o(1)$ for all $0<c<1/2$.\\
Also, straightforward estimation shows that $\int_{cx}^{(1-c)x} \frac{\oo{F}(x-y)F(dy)}{\oo{F}(x)}=O(x^{-\alpha})$, and hence condition (D3iii) is met.  We take $B(x)=\Omega$ for all $x \geq 0$ so that there is nothing to show for (D3i).

We now consider condition (D4). For any $i \geq 1$, and any $H(x)$ satisfying \eqref{long-tail},
\begin{align*}
\P(X_i>x+h(x),X_{i+1}<-h(x)) = & \E\((\P(Z_1-W>x+h(x),Z_2-Y_2Z_1<-h(x)|W)\))\\
	\geq & \E \(( \P(Z_1> x + h(x) + W, Z_2<Z_1-h(x)|W)\))\\
	\geq & \E\(( \P(Z_1>x+h(x)+W, Z_2<x+W|W)\))\\
	= & \E\((\P(Z_1>x+h(x)+W)\P(Z_2<x+W|W)\)).
	\end{align*}
Hence, by Fatou's lemma,
\begin{align*}
	&\liminf_{x \to \infty} \frac{\P(X_i>x+h(x),X_{i+1}<-h(x))}{\oo{F}(x)}\\
	\geq & \liminf_{x \to \infty} \E \((\frac{\P(Z_1>x+h(x)+W|W)}{\P(Z_1>x)}\P(Z_2<x+W|W)\))\\
	\geq & \E \(( \liminf_{x \to \infty} \((\frac{\P(Z_1>x+h(x)+W|W)}{\P(Z_1>x)}\P(Z_2<x+W|W)\))\))\\
	= & 1,
\end{align*}
and so condition (D4) is not met.\\
Finally, we show that the conclusion of Proposition \ref{Genrv} fails in this example:
\begin{align*}
\P(X_1+ \ldots +X_n>x) = & \P(Z_n+(1-Y_n)Z_{n-1}+ \ldots + (1-Y_2)Z_1-Y_1Z_0>x)\\
	\leq & \P(Z>x).
\end{align*}

\section{Notations and Definitions}
\subsection{Notation}
For a random variable (r.v.) $X$ with distribution function (d.f.) $F$ we denote its tail distribution $\P(X>x) \defn \oo{F}(x)$.  For two independent r.v.'s $X$ and $Y$ with d.f.'s $F$ and $G$ we denote the convolution of $F$ and $G$ by $F*G(x) \defn \int_{-\infty}^{\infty} F(x-y)G(dy) = \P(X+Y \leq x)$, and the $n$-fold convolution $F * \cdots *F \defn F^{*n}$.  In the case of non-negative random variables we have $F*G(x) \defn \int_{0}^{x} F(x-y)G(dy) = \P(X+Y \leq x)$.

Throughout, unless stated otherwise, all limit relations are for $x \to \infty$.  Let $a(x)$ and $b(x)$ be two positive functions such that
\[l_1 \leq \liminf_{x \to \infty} \frac{a(x)}{b(x)} \leq \limsup_{x \to \infty} \frac{a(x)}{b(x)} \leq l_2.\]
We write $a(x)=O(b(x))$ if $l_2< \infty$ and $a(x)=o(b(x))$ if $l_2=0$.  We say that $a(x)$ and $b(x)$ are weakly equivalent, written $a(x) \asymp b(x)$, if both $l_1 > 0$ and $l_2<\infty$, and that $a(x)$ and $b(x)$ are (strongly) equivalent, written $a(x) \sim b(x)$, if $l_1=l_2=1$.

For any event $A \in \Omega$ we define the indicator function of event $A$ as
\[\1(A)= \left\{	\begin{array}{cc}  
1 & \textrm{if } A \textrm{ occurs},\\
	0 &  \textrm{otherwise }. \end{array} \right.\]
	
We use $X\vee Y$ to mean $\max(X,Y)$, and $X \wedge Y$ to mean $\min(X,Y)$.

\subsection{Definitions and Basic properties}
We recall some well-known results on long-tailed and subexponential distributions and their sub-classes.

A  non-negative r.v. $X$ with distribution $F$ is heavy-tailed if $\E(e^{\gamma X}) = \infty$ for all $\gamma>0$, and long-tailed if $\oo{F}(x+1) \sim \oo{F}(x)$.  The class of long-tailed distributions, $\LL$, is a proper sub-class of the heavy-tailed distributions.  The distribution $F$ is long-tailed if and only if there exists a positive function $h$, monotonically increasing to zero and satisfying $h(x)<x$ such that \begin{equation}\label{LT2}\oo{F}(x-h(x)) \sim \oo{F}(x),\end{equation}
A  r.v. $X$ with distribution $F$  is subexponential if $\oo{F}^{*2}(x) \sim 2\oo{F}(x)$.  This is equivalent to $\oo{F}^{*n}(x) \sim n\oo{F}(x)$ for any $n \geq 1$.  The class of subexponential distributions, $\SS$, is a proper sub-class of the class $\LL$ of long-tailed distributions.  If $X_1$ and $X_2$ are independent and have common d.f. $F$ then $\P(X_1+X_2>x) \sim \P(\max(X_1,X_2)>x)$.  

A  non-negative r.v. $X$ with distribution $F$ supported on the positive half-line is subexponential if and only if 
\begin{enumerate}
\item  $F$ is long-tailed;  
\item for any $h(x)<x/2$ tending monotonically to infinity,
\begin{equation}
\int_{h(x)}^{x-h(x)}\oo{F}(x-y)F(dy) = o(\oo{F}(x))
\end{equation}
\end{enumerate}

A positive function $l$ is slowly varying if, for all $\lambda > 0$, $l(\lambda x) \sim l(x)$.  A distribution function $F$ belongs to the class of regularly varying distributions of degree $\alpha$, $\RR_{-\alpha}$, if $\oo{F}(x) \sim l(x) x^{-\alpha}$ for some slowly varying function $l$.  A distribution function $F$ belongs to the class of extended regular varying distributions, ERV, if $\liminf_{x \to \infty} \frac{\oo{F}(\lambda x)}{\oo{F}(x)} \geq \lambda^{-c}$ for some $c \geq 0$ and all $\lambda \geq 1$.   A distribution function $F$ belongs to the class of intermediately regular varying distributions, IRV, also called consistent variation by some authors, if $\lim_{\lambda \downarrow 1} \liminf_{x \to \infty} \frac{\oo{F}(\lambda x)}{\oo{F}(x)} = 1$.  A distribution function $F$ belongs to the class of dominatedly regular varying distributions, $\DD$, if $\liminf_{x \to \infty} \frac{\oo{F}(\lambda x)}{\oo{F}(x)} \geq 0$ for some $\lambda > 1$.  A distribution function $F$ belongs to the class of rapidly varying distributions, $\RR_{-\infty}$, if $\lim_{x \to \infty} \frac{\oo{F}(\lambda x)}{\oo{F}(x)} = 0$ for all $\lambda \geq 1$.  We have the proper inclusions (see \cite{EKM})
\[\RR_{-\alpha} \subset \textrm{ERV} \subset \textrm{IRV} \subset \DD \cap \LL \subset \SS \subset \LL.\]

The following three lemmas are due to \cite{CS, TT, LGH} and are used in the development of Example 4 in Section 4.

\begin{lemma}\label{6.1}
The class $\RR_{- \infty}$ is closed under product convolution.
\end{lemma}

\begin{lemma}\label{6.2}
Let $X$ and $Y$ be two independent positive r.v.s, and let the distribution function of $X$, $F_X \in \DD \cap \LL$, and that of $Y$, $F_Y \in \RR_{- \infty}$.  Let the distribution function of $XY$ be $F_{XY}$.  Then $\oo{F}_{XY}(x) \asymp \oo{F}_X(x)$.
\end{lemma}

\begin{lemma}\label{6.3}
If $F \in \DD$ and $F_{\eta} \in \RR_{- \infty}$ then there exists $\epsilon>0$ such that $\oo{F_{\eta}}(x^{1-\epsilon}) = o(\oo{F}(x))$.
\end{lemma}

\vspace{0.5cm}

\bfseries\large{Acknowledgment}

\mdseries\normalsize
The authors are grateful to Stan Zachary for many fruitful and insightful conversations.  They are also grateful to the associate editor and anonymous referees for many helpful and perceptive comments which have greatly improved the paper.


\begin{thebibliography}{99}
\bibitem{DFK}
D. Denisov, S. Foss and D. Korshunov,
Tail asymptotics for the supremum of a random walk when the mean is not finite, 
Queueing Systems 46 (2004) 15–33.
\end{thebibliography}
\end{document}